\documentclass[12pt,a4paper]{article}
\usepackage{amsfonts}
\usepackage{amssymb}
\usepackage{mathrsfs}
\usepackage{amsmath}
\usepackage{amsthm}
\usepackage{enumerate}

\allowdisplaybreaks
\newtheorem{defn}{Definition}[section]
\newtheorem{thm}[defn]{Theorem}
\newtheorem{lem}[defn]{Lemma}
\newtheorem{prop}[defn]{Proposition}
\newtheorem{cor}[defn]{Corollary}
\newtheorem{ex}[defn]{Example}
\newtheorem{re}[defn]{Remark}
\def\K{{\bf K}}

\def\ad{{\rm ad}}
\def\dim{{{\rm dim}}}

\def\ch{{{\rm ch}}}
\def\Der{\rm Der}
\def\Id{{{\rm Id}}}

\begin{document}
\title{{\bf Representations and $T$*-extensions of hom-Jordan-Lie algebras}}
\author{\normalsize \bf Jun Zhao,  Liangyun Chen, Lili Ma}
\date{{{\small{School of Mathematics and Statistics, Northeast Normal University,  \\Changchun 130024,  China
 }}}} \maketitle
\date{}
{\bf\begin{center}{Abstract}\end{center}}

 The purpose of this paper is to study representations and $T$*-extensions of hom-Jordan-Lie algebras.
  In particular, adjoint representations, trivial representations, deformations and many properties of $T$*-extensions
  of hom-Jordan-Lie algebras are studied in detail. Derivations and central extensions of hom-Jordan-Lie algebras are
  also discussed as an application.



\renewcommand{\thefootnote}{\fnsymbol{footnote}}
\footnote[0]{ Corresponding author(L. Chen): chenly640@nenu.edu.cn.}
\footnote[0]{ Supported by  NNSF of China (No. 11171055),  Natural
Science Foundation of  Jilin province (No. 201115006), Scientific
Research Foundation for Returned Scholars
    Ministry of Education of China. }

\section{Introduction}
 The definition of hom-Lie algebras was introduced by Hartwig,  Larsson
and silverstory [4] to character the structures on deformations of
the Witt and the Virasoro algebras. If the skew-symmetric bracket of
the hom-Lie algebra $(L, [\cdot, \cdot]_L,\alpha)$ is replaced by
Jordan-symmetric bracket, then the triple $(L, [\cdot,
\cdot]^{'},\alpha)$ is called a hom-Jordan-Lie algebra. It is clear
that the hom-Jordan-Lie algebra $(L, [\cdot, \cdot]_L, {\rm id})$ is
the Jordan-Lie algebra $L$ itself. Recently, hom-Lie algebras were
generalized to hom-Lie superalgebras by Ammar and Makhlouf.
Quadratic hom-Lie algebras were studied first in [5]. Then the study
was extended to color hom-Lie algebras in [1].  More applications of
the hom-Lie algebras, hom-algebras and hom-Lie superalgebras can be
found in [1-8, 11].

 The notion of Jordan-Lie algebras was introduced in [9], which is intimately related to both Lie and Jordan superalgebras.
Engel's theorem of Jordan-Lie algebras was proved, and some
properties of Cartan subalgebras of Jordan-Lie algebras were given
in [10]. The purpose of this paper is to study representations and
$T$*-extensions of hom-Jordan-Lie algebras.
  In particular, adjoint representations, trivial representations, deformations and many properties of $T$*-extensions
  of hom-Jordan-Lie algebras are studied in detail. Derivations and central extensions of hom-Jordan-Lie algebras are
  also discussed as an application.

The paper is organized as follows. In Section 2   we give the definition of hom-Jordan-Lie algebras, and show that the direct sum of
two hom-Jordan-Lie algebras is still a hom-Jordan-Lie algebra. A linear map between hom-Jordan-Lie algebras is a morphism if and only if
its graph is a hom subalgebra. In section 3 we study derivations of multiplicative hom-Jordan-Lie algebras. For any nonnegative integer $k$,
we define $\alpha^{k}$-derivations of multiplicative hom-Jordan-Lie algebras. Considering the direct sum of the space of $\alpha$-derivations, we prove that it is a Lie algebra (Proposition 3.3). In particular, any $\alpha$-derivation gives rise to a derivation extension of the multiplicative hom-Jordan-Lie algebra $(L, [\cdot, \cdot]_L, \alpha)$ (Theorem 3.5). In Section 4 we give the definition of representations of multiplicative hom-Jordan-Lie algebras. We can obtain the semidirect product multiplicative hom-Jordan-Lie algebra $(L\oplus V, [\cdot, \cdot]_{\rho_{A}}, \alpha+A)$ associated to any representation $\rho_{A}$ on $V$ of the multiplicative hom-Jordan-Lie algebra $(L, [\cdot, \cdot]_L, \alpha)$ (Proposition 4.4). In Section 5 we study trivial representations of multiplicative hom-Jordan-Lie algebras. We show that central extensions of a multiplicative hom-Jordan-Lie algebra are controlled by the second cohomology with coefficients in the trivial representation (Theorem 5.1). In Section 6 we study the adjoint representation of a regular hom-Jordan-Lie algebra $(L, [\cdot, \cdot]_L, \alpha)$. For any integer $s$, we define the $\alpha^{k+1}$-derivations. We show that a 1-cocycle associated to the $\alpha ^{k}$-derivation is exactly an  $\alpha^{k+1}$-derivation of the regular hom-Jordan-Lie algebra $(L, [\cdot, \cdot]_L, \alpha)$ in some conditions (Proposition 6.3). We also give the definition of hom-Nijienhuis operators of regular hom-Jordan-Lie algebras. We show that the deformation generated by a hom-Nijienhuis operator is trivial. In Section 7 we study $T$*-extensions of hom-Jordan-Lie algebras, show that $T^*$-extensions preserve many properties such as nilpotency, solvability and decomposition in some sense.
\section{hom-Jordan-Lie algebras}
\begin{defn}$[9]$
A Jordan-Lie algebra is a couple  $(L, [\cdot, \cdot]_L)$ consisting of a  vector space $L$ and a bilinear map (bracket) $[ \cdot, \cdot]_L:L\times L\rightarrow L$ satisfying
\begin{equation}[x, y]=-\delta[y, x] ,\delta=\pm 1,\end{equation}
$$[x, [y, z]]+[y, [z, x]]+[z, [x, y]]=0. \qquad\forall x, y, z\in\L.$$
\end{defn}
\begin{defn}
(1) A hom-Jordan-Lie algebra is a triple  $(L, [\cdot, \cdot]_L, \alpha)$ consisting of a  vector space $L$,  a bilinear map (bracket) $[\cdot, \cdot]_L:L\times L\rightarrow L$ and a map $\alpha:L\rightarrow L$ satisfying
\begin{equation}[x, y]=-\delta[y, x] ,\delta=\pm 1,\end{equation}
$$[\alpha(x), [y, z]]+[\alpha(y), [z, x]]+[\alpha(z), [x, y]]=0, \qquad\forall x, y, z\in\L.$$

(2) A hom-Jordan-Lie algebra is multiplicative  if $\alpha$ is an algebra morphism, i.e. for any $x,  y,  \in L$, we have $\alpha([x, y]_L)=[\alpha(x), \alpha(y)]_L;$

(3) A hom-Jordan-Lie algebra is regular if $\alpha$ is an algebra automorphism;

(4) A subvector space $\eta\subseteq L$ is a hom subalgebra of $(L, [\cdot, \cdot]_L, \alpha)$ if $\alpha(\eta)\subseteq\eta$  and$$[x, y]_L\in\eta,  \qquad\forall x, y\in\eta;$$

(5) A subvector space $\eta\subseteq L$ is a hom ideal of $(L, [\cdot, \cdot]_L, \alpha)$ if $\alpha(\eta)\subseteq\eta$ and $$[x, y]_L\in\eta,\qquad\forall x\in\eta,y\in\L.$$
\end{defn}
\begin{defn}
A $\delta$-hom associative algebra is a couple  $(L,\alpha)$ consisting of a  vector space $L$,  a bilinear map on $L$ and a map $\alpha:L\rightarrow L$ satisfying
\begin{equation}(\lambda x)y=x(\lambda y)=\lambda(xy),\delta=\pm 1,\end{equation}
$$\alpha(x)(yz)=\delta(xy)\alpha(z).\qquad\forall x, y, z\in\L.$$
\end{defn}
\begin{prop}\label{proposition2.1}
A couple  $(L,\alpha)$ is a $\delta$-hom associative algebra, define a bilinear map (bracket) $[\cdot, \cdot]_L:L\times L\rightarrow L$ satisfying
\begin{equation}[x,y]=xy-\delta yx.
\end{equation}
Then $(L, [\cdot, \cdot]_L, \alpha)$ is a hom-Jordan-Lie algebra.
\end{prop}
\begin{proof}
\begin{eqnarray*}
&&[\alpha(x),[y,z]]+c.p.(x,y,z)=[\alpha(x),yz-\delta zy]+c.p.(x,y,z)\\
&=&\alpha(x)(yz)-\delta(yz)\alpha(x)-\delta\alpha(x)(zy)+(zy)\alpha(x)\\
&&+\alpha(y)(zx)-\delta(zx)\alpha(y)-\delta\alpha(y)(xz)+(xz)\alpha(y)\\
&&+\alpha(z)(xy)-\delta(xy)\alpha(z)-\delta\alpha(z)(yx)+(yx)\alpha(z),
\end{eqnarray*}
\begin{eqnarray*}
\alpha(x)(yz)=\delta(xy)\alpha(z),(zy)\alpha(x)=\delta\alpha(z)(yx);\\
\alpha(y)(zx)=\delta(yz)\alpha(x),(xz)\alpha(y)=\delta\alpha(x)(zy);\\
\alpha(z)(xy)=\delta(zx)\alpha(y),(yx)\alpha(z)=\delta\alpha(y)(xz).
\end{eqnarray*}
Then $[\alpha(x),[y,z]]+c.p.(x,y,z)=0$.
\end{proof}
\begin{prop}
Given two hom-Jordan-Lie algebras $(L, [\cdot, \cdot]_L, \alpha)$ and $(\Gamma, [\cdot, \cdot]_\Gamma, \beta)$, there is a  hom-Jordan-Lie algebra $(L\oplus\Gamma, [\cdot, \cdot]_{L\oplus\Gamma}, \alpha+\beta)$, where the bilinear map $[\cdot, \cdot]_{L\oplus\Gamma}:(L\oplus\Gamma)\times (L\oplus\Gamma)\rightarrow L\oplus\Gamma$ is given by
$${[u_1+v_1, u_2+v_2]}_{L\oplus\Gamma}={[u_1, u_2]}_L+{[v_1, v_2]}_\Gamma,  \qquad\forall  u_1\in L, u_2\in L,  \forall  v_1\in \Gamma, v_2\in \Gamma,  $$
and the linear map $(\alpha+\beta):L\oplus\Gamma \rightarrow L\oplus\Gamma$ is given by
$$(\alpha+\beta)(u+v)=\alpha(u)+\beta(v), \qquad\forall  u\in L,  v\in \Gamma.$$
 \end{prop}
\begin{proof}
 For any  $u_i\in L,   v_i\in \Gamma, $ we have
$${[u_1+v_1, u_2+v_2]}_{L\oplus\Gamma}={[u_1, u_2]}_L+{[v_1, v_2]}_\Gamma, $$
\begin{eqnarray*}
-\delta{[u_2+v_2, u_1+v_1]}_{L\oplus\Gamma}=-\delta({[u_2, u_1]}_L+{[v_2, v_1]}_\Gamma)
={[u_1, u_2]}_L+{[v_1, v_2]}_\Gamma.
\end{eqnarray*}
By a direct computation, we have
\begin{eqnarray*}
&&[(\alpha+\beta)(u_1+v_1), [u_2+v_2, u_3+v_3]_{L\oplus\Gamma}]_{L\oplus\Gamma}+c.p.((u_1+v_1), (u_2+v_2), (u_3+v_3))\\
&=&[\alpha(u_1)+\beta(v_1), [u_2, u_3]_{L}+[v_2, v_3]_{\Gamma}]_{L\oplus\Gamma}+c.p.\\
&=&[\alpha(u_1), [u_2, u_3]_L]_L+c.p.(u_1, u_2, u_3)+[\beta(v_1), [v_2, v_3]_\Gamma]_\Gamma+c.p.(v_1, v_2, v_3)\\
&=&0,
\end{eqnarray*}
where $c.p.(a, b, c)$ means the cyclic permutations of $a,  b,  c$.
\end{proof}
\begin{defn}Let $(L, [\cdot, \cdot]_L, \alpha)$ and $(\Gamma, [\cdot, \cdot]_\Gamma, \beta)$ be two hom-Jordan-Lie algebras. A map $\phi:L \rightarrow \Gamma$ is said to be a morphism of hom-Jordan-Lie algebras if
\begin{equation} \phi[u, v]_L=[\phi(u), \phi(v)]_\Gamma, \qquad\forall  u, v\in L, \end{equation}
\begin{equation}\phi\circ\alpha=\beta\circ\phi. \qquad\qquad\qquad\qquad\end{equation}

Denote by $\mathfrak{G}_\phi\in L\oplus\Gamma$ is the graph of a linear map $\phi:L \rightarrow \Gamma$.
\end{defn}
\begin{prop}\label{proposition2.1}
A map $\phi:(L, [\cdot, \cdot]_L, \alpha)\rightarrow(\Gamma, [\cdot, \cdot]_\Gamma, \beta)$ is a morphism of hom-Jordan-Lie algebras if and only if the graph $\mathfrak{G}_\phi\in L\oplus\Gamma$ is a hom subalgebra of  $(L\oplus\Gamma, [\cdot, \cdot]_{L\oplus\Gamma}, \alpha+\beta)$.
\end{prop}
\begin{proof}
Let $\phi:(L, [\cdot, \cdot]_L, \alpha)\rightarrow(\Gamma, [\cdot, \cdot]_\Gamma, \beta)$ be a morphism of hom-Jordan-Lie algebras, then for any $u, v\in L$, we have
$$[u+\phi(u), v+\phi(v)]_{L\oplus\Gamma}=[u, v]_{L}+[\phi(u), \phi(v)]_{\Gamma}=[u, v]_{L}+\phi[u, v]_L.$$
Thus the graph $\mathfrak{G}_\phi$ is closed under the bracket operation $[\cdot, \cdot]_{L\oplus\Gamma}.$ Furthermore, by (9), we have
$$(\alpha+\beta)(u+\phi(u))=\alpha(u)+\beta\circ\phi(u)=\alpha(u)+\phi\circ\alpha(u),$$
which implies that $(\alpha+\beta)(\mathfrak{G}_\phi)\subset \mathfrak{G}_\phi.$ Thus $\mathfrak{G}_\phi$ is a hom subalgebra of $(L\oplus\Gamma, [\cdot, \cdot]_{L\oplus\Gamma}, \alpha+\beta)$.

Conversely,  if the graph $\mathfrak{G}_\phi\subset L\oplus\Gamma$ is a hom subalgebra of
$(L\oplus\Gamma, [\cdot, \cdot]_{L\oplus\Gamma}, \alpha+\beta)$, then we have
$$[u+\phi(u), v+\phi(v)]_{L\oplus\Gamma}=[u, v]_{L}+[\phi(u), \phi(v)]_{\Gamma}\in\mathfrak{G}_\phi, $$
which implies that $$[\phi(u), \phi(v)]_{\Gamma}=\phi[u, v]_{L}.$$
Furthermore,  $(\alpha+\beta)(\mathfrak{G}_{\phi})\subset \mathfrak{G}_{\phi}$ yields that
$$(\alpha+\beta)(u+\phi(u))=\alpha(u)+\beta\circ\phi(u)\in\mathfrak{G}_\phi, $$
which is equivalent to the condition $\beta\circ\phi(u)=\phi\circ\alpha(u), $ i.e. $\beta\circ\phi=\phi\circ\alpha.$ Therefore,  $\phi$ is a morphism of hom-Jordan-Lie algebras.
\end{proof}
\section{Derivations of hom-Jordan-Lie algebras}
Let $(L, [\cdot, \cdot]_L, \alpha)$ be a multiplicative hom-Jordan-Lie algebra. For any nonnegative integer $k$, denote by $\alpha^k$ the $k$-times composition of $\alpha$,  i.e.\\
$$\alpha^k=\alpha\circ\cdots\circ\alpha \quad(k-times).$$
In particular, $\alpha^0=\mathrm{id}$ and $\alpha^1=\alpha.$ If $(L, [\cdot, \cdot]_L, \alpha)$ is a regular hom-Jordan-Lie algebra, we denote by $\alpha^{-k}$ the $k$-times composition of $\alpha^1, $ the inverse of $\alpha.$
\begin{defn}
For any nonnegative integer $k$, a linear map $D:L \rightarrow L$ is called an $\alpha^k$-derivation of the multiplicative  hom-Jordan-Lie algebra $(L, [\cdot, \cdot]_L, \alpha)$, if
\begin{equation} [D, \alpha]=0, \quad i.e.\quad{D}\circ\alpha=\alpha\circ{D}, \end{equation}
and
\begin{equation}D[u, v]_L=\delta^{k}([D(u), \alpha^{k}(v)]_L+[\alpha^{k}(u), D(v)]_L), \qquad\forall u, v\in L.\end{equation}
For a regular hom-Jordan-Lie algebra, $\alpha^{-k}$-derivations can be defined similarly.
\end{defn}

Denote by $ \texttt{\Der}_{\alpha^s}(L)$ is the set of $\alpha^s$-derivations of the multiplicative  hom-Jordan-Lie algebra $(L, [\cdot, \cdot]_L, \alpha)$. For any $u\in L$ satisfying $\alpha(u)=u$, define $D_{k}(u):L \rightarrow L$ by\\
$$D_{k}(u)(v)=\delta[u, \alpha^{k}(v)]_L, \,\delta^{k}=1, \quad \forall v\in L.$$
Then $D_{k}(u)$ is an $\alpha^{k}$-derivation.which we call an \textbf{inner} $\alpha^{k+1}$-derivation. In fact,  we have
$$D_{k}(u)(\alpha(v))=\delta[u, \alpha^{k+1}(v)]_L=\alpha(\delta[u, \alpha^{k}(v)]_L)=\alpha\circ D_{k}(u)(v). $$
On the other hand,  we have
\begin{eqnarray*}&&D_{k}u([v, w]_L)=\delta[u, \alpha^{k}([v, w]_L)]_L=\delta[\alpha(u), [\alpha^{k}(v), \alpha^{k}(w)]_L]_L\\
&=&-\delta([\alpha^{k+1}(v), [\alpha^{k}(w), u]_L]_L+[\alpha^{k+1}(w), [u, \alpha^{k}(v)]_L]_L)\\
&=&-\delta[\alpha^{k+1}(v), [\alpha^{k}(w), u]_L]_L-\delta[\alpha^{k+1}(w), [u, \alpha^{k}(v)]_L]_L)\\
&=&\delta^{k+1}[\alpha^{k+1}(v), \delta[u, \alpha^{k}(w)]_L]_L+\delta^{k+1}[\delta[u, \alpha^{k}(v)]_L), \alpha^{k+1}(w)]_L\\
&=&\delta^{k+1}([D_{k}u(v), \alpha^{k+1}(w)]_L+[\alpha^{k+1}(v), D_{k}u(w)]_L).
\end{eqnarray*}
Therefore,  $D_{k}(u)$ is an  $\alpha^{k+1}$-derivation. Denote by $ \mathrm{Inn}_{\alpha^s}($\emph{L}$)$ the set of inner  $\alpha^{k}$-derivations, i.e. \begin{equation}\mathrm{Inn}_{\alpha^{k}}(L)=\{\delta[u, \alpha^{k-1}(\cdot)]_L|u \in L, \alpha(u)=u,\delta^{k}=1\}.\end{equation}

For any $D \in \texttt{\Der}_{\alpha^{k}}(L)$ and $D^{'} \in \texttt{\Der}_{\alpha^{s}}(L), $ define their commutator $[D, D^{'}]$ as usual:
\begin{equation}[D, D^{'}]=D \circ D^{'}-D^{'} \circ D.\end{equation}
\begin{lem}\label{lemma3.2}For any $D \in \texttt{\Der}_{\alpha^{k}}(L)$ and $D^{'} \in \texttt{\Der}_{\alpha^{s}}(L)$, we have\\
$$[D, D^{'}]\in \texttt{\Der}_{\alpha^{k+s}}(L).$$
\end{lem}
\begin{proof}
For any $u, v \in L$, we have
\begin{eqnarray*}&&[D, D^{'}]([u, v]_L)=D\circ D^{'}([u, v]_L)-D^{'}\circ D([u, v]_L)\\
&=&\delta^{s}D([D^{'}(u), \alpha^{s}(v)]_L+[\alpha^{s}(u), D^{'}(v)]_L)\\
&&-\delta^{k}D^{'}([D(u), \alpha^{k}(v)]_L+[\alpha^{k}(u), D(v)]_L)\\
&=&\delta^{s}D[D^{'}(u), \alpha^{s}(v)]_L+\delta^{s}D[\alpha^{s}(u), D^{'}(v)]_L\\
&&-\delta^{k}D^{'}[D(u), \alpha^{k}(v)]_L-\delta^{k}D^{'}[\alpha^{k}(u), D(v)]_L\\
&=&\delta^{k+s}([D\circ D^{'}(u), \alpha^{k+s}(v)]_L+[\alpha^{k}\circ D^{'}(u), D\circ\alpha^{s}(v)]_L\\
&&+[D\circ \alpha^{s}(u), \alpha^{k} \circ D^{'}(v)]_L+[\alpha^{k+s}\circ D(u), D \circ D^{'}(v)]_L\\
&&-[D^{'} \circ D(u), \alpha^{k+s}(v)]_L-[\alpha^{s}\circ D(u), D^{'}\circ\alpha^{k}(v)]_L\\
&&-[D^{'}\circ \alpha^{k}(u), \alpha^{s} \circ D(v)]_L-[\alpha^{k+s}(u), D ^{'} \circ D(v)]_L).
\end{eqnarray*}
Since $D$ and $D^{'}$ satisfy
$$D\circ \alpha=\alpha \circ D, \qquad D^{'} \circ \alpha=\alpha \circ D^{'}, $$
we have
$$D\circ \alpha^{s}=\alpha^{s} \circ D, \qquad D^{'} \circ \alpha^{k}=\alpha^{k} \circ D^{'}.$$
Therefore,  we have
\begin{eqnarray*}
[D, D^{'}]([u, v]_L)&=&\delta^{k+s}([D\circ D^{'}(u)-D^{'}\circ D(u), \alpha^{k+s}(v)]_L\\
&&+[\alpha^{k+s}(u), D\circ D^{'}(u)-D^{'}\circ D(v)]_L)\\
&=&\delta^{k+s}([[D, D^{'}](u), \alpha^{k+s}(v)]_L+[\alpha^{k+s}(u), [D, D^{'}](v), ]_L).
\end{eqnarray*}
Furthermore, it is straightforward to see that
\begin{eqnarray*}
[D, D^{'}]\circ \alpha&=&D\circ D^{'}\circ\alpha-D\circ D^{'}\circ\alpha\\
&=&\alpha\circ D\circ D^{'}-\alpha\circ D\circ D^{'}\\
&=&\alpha \circ [D, D^{'}],
\end{eqnarray*}
which yields that $[D, D^{'}]\in \texttt{\Der}_{\alpha^{k+s}}(L)$.
\end{proof}
Denote by \begin{eqnarray}\texttt{\Der}(L)=\bigoplus_{k\geq 0} \texttt{\Der}_{\alpha^{k}}(L).\end{eqnarray}%1ºá
By Lemma 3.2,  obviously we have
\begin{prop}With the above notations,  $ \texttt{\Der}(L)$ is a Lie algebra.\end{prop}
\begin{re}
Similarly,  we can obtain a Lie algebra $\texttt{\Der}(L)=\bigoplus_{k}(\texttt{\Der}_{\alpha^{k}}(L)$,  where $k$ is any integer,  for a regular hom-Jordan-Lie algebra.
\end{re}

At the end of this section,  we consider the derivation extension of the multiplicative hom-Jordan-Lie algebra $(L, [\cdot, \cdot]_L, \alpha)$ and give an application of the $\alpha$-derivation $ \texttt{\Der}_{\alpha}(L)$.

For any  linear map $D:L \rightarrow L$,  consider the vector space $L \oplus RD$. Define a bilinear bracket operation$[\cdot, \cdot]_{L}$ on $L \oplus RD$ by
$$[u+mD, v+nD]_{D}=[u, v]_{L}+mD(v)-\delta{nD(u)}, $$
$$ [u, v]_{D}=[u, v]_{L}, [D, u]_{D}=-\delta[u, D]_{D}=D(u),   \qquad \forall u, v\in L.$$
Define a linear map $\alpha^{'}:L \oplus RD \rightarrow L \oplus RD$ by $\alpha^{'}(u+D)=\alpha(u)+D, $ i.e.\\
\begin{displaymath}
\mathbf{\alpha^{'}} =
\left( \begin{array}{cc}
\alpha & 0 \\
0 & 1  \\
\end{array}\right).
\end{displaymath}
\begin{thm}
With the above notations, $(L \oplus RD,  [\cdot, \cdot]_{D}, \alpha^{'})$ is a multiplicative hom-Jordan-Lie algebra if and only if D is an $\alpha$-derivation of the multiplicative hom-Jordan-Lie algebra $(L, [\cdot, \cdot]_L, \alpha)$, with $(1-\delta)D\circ D=0$.
\end{thm}
\begin{proof}First,  $[\cdot, \cdot]_{D}$ satisfies the condition (3), for any $u,  v,  w\in L$, we have\\
$$[u+mD, v+nD]_{D}=[u, v]_{L}+mD(v)-\delta nD(u), $$and
\begin{eqnarray*}
-\delta[v+nD, u+mD]_{D}&=&-\delta([v, u]_{L}+nD(u)-\delta mD(v))\\
&=&[u, v]_{L}+mD(v)-\delta nD(u)\\
&=&[u+mD, v+nD]_{D}.
\end{eqnarray*}
On the other hand,  we have
\begin{eqnarray*}
&&\alpha^{'}([u+mD, v+nD, ]_{D})=\alpha^{'}([u, v]_{L}+mD(v)-\delta nD(u))\\
&=&\alpha[u, v]_{L}+m\alpha\circ D(v))-\delta n\alpha\circ D(u),
\end{eqnarray*}
and
\begin{eqnarray*}
&&[\alpha^{'}(u+mD), \alpha^{'}(v+nD)]_{D}=[\alpha(u)+mD, \alpha(v)+nD]_{D}\\
&=&[\alpha(u), \alpha(v)]_{L}+mD\circ \alpha(v)-\delta nD\circ \alpha(u).
\end{eqnarray*}
Since $\alpha$ is an algebra morphism, $\alpha^{'}$ is an algebra morphism if and only if
$$D\circ \alpha=\alpha\circ D.$$
By a direct calculation,  we have
\begin{eqnarray*}
&&[\alpha^{'}(D), [u, v]_{D}]_{D}+[\alpha^{'}(u), [v, D]_{D}]_{D}+[\alpha^{'}(v), [D, u]_{D}]_{D}\\
&=&D[u, v]_{D}-\delta[\alpha(u), D(v)]_{D}+[\alpha(v), D(u)]_{D}\\
&=&D[u, v]_{D}-\delta[\alpha(u), D(v)]_{D}-\delta[D(u), \alpha(v)]_{D}.
\end{eqnarray*}
Therefore, if the hom-Jordan-Jacobi identity is satisfied, then the following condition holds
$$D[u, v]_{D}-\delta[\alpha(u), D(v)]_{D}-\delta[D(u), \alpha(v)]_{D}=0.$$
On the other hand, if $D$ is an $\alpha$-derivation of $(L, [\cdot, \cdot]_L, \alpha)$, and $(1-\delta)D\circ D=0.$
\begin{eqnarray*}
&&[\alpha^{'}(u+mD), [v+nD, w+pD]_{D}]_{D}+c.p(u+mD,v+nD,w+pD)\\
&=&[\alpha(u)+m D,[v,w]_{L}+n D(w)-\delta p D(v)]_{D}+c.p(u+mD,v+nD,w+pD)\\
&=&[\alpha(u), [v, w]_{L}]_{L}+n[\alpha(u),D(w)]_{L}-\delta p[\alpha(u),D(v)]_{L}\\
&+&m D[v,w]_{L}+mn D\circ D(w)-\delta pm D\circ D(v)+c.p(u+mD,v+nD,w+pD)\\
&=&(1-\delta)(np D\circ D(u)+pm D\circ D(v)+mn D\circ D(w))=0.
\end{eqnarray*}
Thus $(L \oplus RD,  [\cdot, \cdot]_{D}, \alpha^{'})$ is a multiplicative hom-Jordan-Lie algebra if and only if $D$ is an $\alpha$-derivation of $(L, [\cdot, \cdot]_L, \alpha)$, with $(1-\delta)D\circ D=0$.
\end{proof}
\section{Representations of hom-Jordan-Lie algebras}
In this section we study representations of  multiplicative hom-Jordan-Lie algebras and give corresponding  operations. We also prove that one can form semidirect product  multiplicative hom-Jordan-Lie algebras when given representations of  multiplicative hom-Jordan-Lie algebras. Let $A\in pl(V)$ be an arbitrary linear transformation from $V$ to $V$.
\begin{defn}
A representation of  multiplicative hom-Jordan-Lie algebra $(L, [\cdot, \cdot]_L, \alpha)$ on the  vector space V with respect to  $A\in pl(V)$ is a linear map $\rho_ A:L \rightarrow pl(V)$,  such that for any $ u, v\in L$,  the basic property of this map is that:
\begin{eqnarray}
&&\rho_{A}(\alpha(u))\circ A=A\circ\rho_{A}(u),\\
&&\rho_{A}([u, v]_L)\circ A=\rho_{A}(\alpha(u))\circ\rho_{A}(v)-\delta\rho_{A}(\alpha(v))\circ\rho_{A}(u).
\end{eqnarray}

The set of k-cochains on L with values in V,  which we denote by $C^{k}(L;V)$,  is the set of k-linear maps from $L \times \cdots \times L$ (k-times) to V:
\begin{eqnarray*}
C^{k}(L;V)\triangleq \{f:L \times \cdots \times L (k-times)\rightarrow V\; is\; a\; linear \; map\}.
\end{eqnarray*}

A k-hom-cochain on L with values in V is defined to be a k-cochain $f\in C^{k}(L;V)$ such that it is compatible with $\alpha$ and A in the sense that $A\circ f=f\circ \alpha$,  i.e.
$$A(f(u_{}, \cdots, u_{k+1}))=f(\alpha(u_{1}), \cdots, \alpha(u_{k+1})).$$
Denote by $C_{\alpha, A}^{k}(L;V)$ the set of  k-hom-cochains:
$$C_{\alpha, A}^{k}(L, V)\triangleq \{f\in C^{k}(L, V)|A\circ f=f\circ \alpha\}.$$

Next we define, the linear map $d_{\rho_{A}}^{k}:C_{\alpha, A}^{k}(L, V) \rightarrow C_{\alpha, A}^{k+1}(L, V)(k=1,2)$ as follows: we set
\begin{eqnarray*}
d_{\rho_{A}}^{1}f(u_{1},u_{2})=\rho_{A}(\alpha(u_{1}))f(u_{2})-\delta\rho_{A}(\alpha(u_{2}))f(u_{1})-\delta f([u_{1}, u_{2}]_{L}),
\end{eqnarray*}
\begin{eqnarray*}
d_{\rho_{A}}^{2}f(u_{1},u_{2},u_{3})
&=&\rho_{A}(\alpha^{2}(u_{1}))f(u_{2},u_{3})-\delta\rho_{A}(\alpha^{2}(u_{2}))f(u_{1},u_{3})+\rho_{A}(\alpha^{2}(u_{1}))f(u_{1},u_{2})\\
&&-f([u_{1}, u_{2}]_{L},\alpha(u_{3}))+\delta f([u_{1}, u_{3}]_{L},\alpha(u_{2}))-\delta f([u_{2}, u_{3}]_{L},\alpha(u_{1})).
\end{eqnarray*}
\end{defn}
\begin{lem}With the above notations,  for any $ f\in C_{\alpha, A}^{k}(L;V), $ we have
$$(d_{\rho_{A}}^{k}\circ f)\circ \alpha=A \circ d_{\rho_{A}}^{k}f.$$
Thus we obtain a well-defined map
$$d_{\rho_{A}}^{k}:C_{\alpha, A}^{k}(L;V) \rightarrow C_{\alpha, A}^{k+1}(L;V)$$with k=1,2.
\end{lem}
\begin{prop}
 With the above notations, we have $d_{\rho_{A}}^{2}\circ d_{\rho_{A}}^{1}=0$.
\end{prop}
\begin{proof}
By straightforward computations, we have
\begin{eqnarray*}
&&d_{\rho_{A}}^{2}\circ d_{\rho_{A}}^{1}f(u_{1},u_{2} , u_{3})\\
&=&\sum_{i=1}^{3}(-\delta)^{i+1}\rho_{A}(\alpha^{2}(u_{i}))d_{\rho_{A_{}}}^{1}f(u_{1}, \cdots, \hat{{u_{i}}} \cdots, u_{3})\\
&+&\delta^{3}\sum_{i<j}(-\delta)^{i+j}d_{\rho_{A}}^{1}f([u_{i}, u_{j}]_{L}, \alpha(u_{1}), \cdots, \hat{\alpha(u_{i})}, \cdots, \hat{\alpha(u_{i})}, \cdots, \alpha(u_{3}))\\
&=&\rho_{A}(\alpha^{2}(u_{1}))d_{\rho_{A}}^{1}f(u_{2}, u_{3})-\delta\rho_{A}(\alpha^{2}(u_{2}))d_{\rho_{A}}^{1}f(u_{1}, u_{3})\\
&+&\rho_{A}(\alpha^{2}(u_{3}))d_{\rho_{A}}^{1}f(u_{1}, u_{2})-d_{\rho_{A}}^{1}f([u_{1},u_{2}], \alpha(u_{3}))\\
&+&\delta d_{\rho_{A}}^{1}f([u_{1},u_{3}], \alpha(u_{2}))-d_{\rho_{A}}^{1}f([u_{2},u_{3}], \alpha(u_{1}))\\
&=&\rho_{A}(\alpha^{2}(u_{1}))(\rho_{A}(\alpha(u_{2}))f(u_{3})-\delta\rho_{A}(\alpha(u_{3}))f(u_{2})-\delta f([u_{2},u_{3}]))\\
&-&\delta\rho_{A}(\alpha^{2}(u_{2}))(\rho_{A}(\alpha(u_{1})f(u_{3}))-\delta\rho_{A}(\alpha(u_{3}))f(u_{1})-\delta f([u_{1},u_{3}]))\\
&+&\rho_{A}(\alpha^{2}(u_{3}))(\rho_{A}(\alpha(u_{1})f(u_{2}))-\delta\rho_{A}(\alpha(u_{2}))f(u_{1})-\delta f([u_{1},u_{2}]))\\
&-&(\rho_{A}(\alpha([u_{1},u_{2}]))f(\alpha(u_{3}))-\delta\rho_{A}(\alpha^{2}(u_{3}))f([u_{1},u_{2}])-\delta f([[u_{1},u_{2}],\alpha(u_{3})]))\\
&+&\delta(\rho_{A}(\alpha([u_{1},u_{3}]))f(\alpha(u_{2}))-\delta\rho_{A}(\alpha^{2}(u_{2}))f([u_{1},u_{3}])-\delta f([[u_{1},u_{3}],\alpha(u_{2})]))\\
&-&(\rho_{A}(\alpha([u_{2},u_{3}]))f(\alpha(u_{1}))-\delta\rho_{A}(\alpha^{2}(u_{1}))f([u_{2},u_{3}])-\delta f([[u_{2},u_{3}],\alpha(u_{1})]))\\
&&                             \rho_{A}(\alpha([u_{2},u_{3}]f(\alpha(u_{1}))=\rho_{A}(\alpha([u_{2},u_{3}]))A\circ f(u_{1})\\
&=&\rho_{A}(\alpha^{2}(u_{2}))\rho_{A}(\alpha(u_{3}))f(u_{1})-\delta\rho_{A}(\alpha^{2}(u_{3}))\rho_{A}(\alpha(u_{2}))f(u_{1}).
\end{eqnarray*}
Then $d_{\rho_{A}}^{2}\circ d_{\rho_{A}}^{1}f(u_{1},u_{2} , u_{3})=0$.
\end{proof}
Associated to the representation $\rho_{A}$,  we obtain the complex $(C_{\alpha, A}^{k}(L;V), d_{\rho_{A}})$. Denote the set of closed  $k$-hom-cochains by $Z^{k}(L;\rho_{A})$ and the set of exact $k$-hom-cochains by $B^{k}(L, \rho_{A})$, $k=1,2$.
\begin{prop}
Given a representation $\rho _{A}$ of the multiplicative hom-Jordan-Lie algebra $(L, [\cdot, \cdot]_L, \alpha)$ on the vector space V with respect to $A\in pl(V) $. Define a bilinear bracket operation $[\cdot, \cdot]_{\rho_{A}}:(L \oplus V)\times (L \oplus V)\rightarrow L \oplus V$ by
\begin{equation}
[u+X, v+Y]_{\rho_{A}}=[u, v]_{L}+\delta\rho_{A}(u)(Y)-\rho_{A}(v)(X).
\end{equation}
Define $\alpha+A:L\oplus V \rightarrow L \oplus V$ by
$(\alpha+A)(u+Y)=\alpha(u)+AX$.\\
Then $(L\oplus V, [\cdot, \cdot]_{\rho_{A}}, \alpha+A)$ is a multiplicative hom-Jordan-Lie algebra,  which we call the semidirect product of the multiplicative hom-Jordan-Lie algebra $(L, [\cdot, \cdot]_L, \alpha)$ and V.
\end{prop}
\begin{proof}
First we show that $[\cdot, \cdot]_{\rho_{A}}$ satisfies Jordan symmetry,
\begin{eqnarray*}
&&-\delta[v+Y, u+X]_{\rho_{A}}\\
&=&-\delta([v, u]_{L}+\delta\rho_{A}(v)X-\rho_{A}(u)Y)\\
&=&[u, v]_{L}+\delta\rho_{A}(u)Y-\rho_{A}(v)X\\
&=&[u+X, v+Y]_{\rho_{A}}.
\end{eqnarray*}
Next we show that $\alpha+A$ is an algebra morphism. On the one hand,  we have
\begin{eqnarray*}
(\alpha+A)[u+X, v+Y]_{\rho_{A}}&=&(\alpha+A)([u, v]_{L}+\delta\rho_{A}(u)(Y)-\rho_{A}(v)(X))\\
&=&\alpha([u, v]_{L})+\delta A\circ\rho_{A}(u)(Y)-A\circ \rho_{A}(v)(X).
\end{eqnarray*}
On the other hand,  we have
\begin{eqnarray*}
&&[(\alpha+A)(u+X), (\alpha+A)(v+Y)]_{\rho_{A}}\\
&=&[\alpha(u)+AX, \alpha(v)+AY]_{\rho_{A}}\\
&=&[\alpha(u), \alpha(v)]_{L}+\delta\rho_{A}(\alpha(u))(AY)-\rho_{A}(\alpha(v))(AX).
\end{eqnarray*}
Since $\alpha$ is an algebra morphism,  $\rho_A$ and $A$ satisfy $(15)$,  it follows that $\alpha+A$ is an algebra morphism with respect to the bracket to $[\cdot, \cdot]_{\rho_{A}}$. It is not hard to deduce that
\begin{eqnarray*}
&&[(\alpha+A)(u+X), [v+Y, w+Z]_{\rho_{A}}]_{\rho_{A}}+c.p.(u+X, v+Y, w+Z)\\
&=&[\alpha(u), [v, w]_{L}]_{L}+c.p.(u, v, w)\\
&&+\rho_{A}(\alpha(u))\circ\rho_{A}(v)(Z)-\delta\rho_{A}(\alpha(u))\circ\rho_{A}(w)(Y)-\rho_{A}([v,w])\circ(AX)\\
&&+\rho_{A}(\alpha(v))\circ\rho_{A}(w)(X)-\delta\rho_{A}(\alpha(v))\circ\rho_{A}(u)(Z)-\rho_{A}([w,u])\circ(AY)\\
&&+\rho_{A}(\alpha(w))\circ\rho_{A}(u)(Y)-\delta\rho_{A}(\alpha(w))\circ\rho_{A}(v)(X)-\rho_{A}([u,v])\circ(AZ),\\
&&      \rho_{A}([v,w])\circ(AX)=\rho_{A}(\alpha(v))\circ\rho_{A}(w)(X)-\delta\rho_{A}(\alpha(w))\circ\rho_{A}(v)(X).
\end{eqnarray*}
By (16),  the hom-Jordan-Jacobi identity is satisfied. Thus,  $(L\oplus V, [\cdot, \cdot]_{\rho_{A}}, \alpha+A)$ is a multiplicative hom-Jordan-Lie algebra.
\end{proof}
\section{The trivial representation of hom-Jordan-Lie algebras}
In this section,  we study the trivial representation of multiplicative hom-Jordan-Lie algebras. As an application, we show that the central extension of a  multiplicative hom-Jordan-Lie algebra $(L, [\cdot, \cdot]_L, \alpha)$ is controlled by the second cohomology of $L$ with coeffcients in the trivial representation.

Now let $V=R$,  Then we have $gl(V)=R$. Any $A\in gl(V)$ is exactly a real number,  which we use the notation $r$. Let $\rho:L\rightarrow gl(v)=R$ be the zero map. Obviously,  $\rho$ is a representation of the multiplicative hom-Jordan-Lie algebra $(L, [\cdot, \cdot]_L, \alpha)$ with respect to any $r\in R$. We will always assume that $r=1$. We call this representation the \textbf{trivial representation} of the  multiplicative hom-Jordan-Lie algebra $(L, [\cdot, \cdot]_L, \alpha)$.

Associated to the trivial representation,  the set of $k$-cochains on $L$,  which we denote by $C^{k}(V)$,  is the set of $k$-linear maps from $V \times \cdots \times V$ to $R$. The set of $k$-cochains is given by
$$C_{\alpha, A}^{k}(V)=\{f\in C^{k}(V)|f\circ\alpha=f\}.$$
The corresponding  operator $d_T:C_{\alpha, A}^{k}(V)\rightarrow C_{\alpha, A}^{k+1}(V)(k=1,2)$ is given by
\begin{equation}
d_{T}^{1}f(u_{1},u_{2})=-\delta f([u_{1}, u_{2}]),
\end{equation}
\begin{equation}
d_{T}^{2}f(u_{1},u_{2},u_{3})=-f([u_{1}, u_{2}],\alpha(u_{3}))+\delta f([u_{1}, u_{3}],\alpha(u_{2}))-f([u_{2},u_{3}],\alpha(u_{1})).
\end{equation}
Denote $Z_{\alpha}^{k}(V)$ and $B_{\alpha}^{k}(V)(k=1,2)$ similarly.

In the following we consider central extensions of the multiplicative hom-Jordan-Lie algebra $(L, [\cdot, \cdot]_L, \alpha)$.
Obviously, $(R, 0, 1)$ is an abelian multiplicative hom-Jordan-Lie algebra with the trivial bracket and the identity morphism. Let $\theta\in  C_{\alpha}^{2}(V)$ and is a Jordan-symmetric($\theta(u, v)=-\delta\theta(v, u)$) map, we consider the direct sum $\eta=L\oplus R$ with the following bracket
\begin{equation}
[u+s, v+t]_{\theta}=[u, v]_{L}+\theta(u, v), \qquad \forall u, v \in L, s, t \in R.
\end{equation}
Define $\alpha_\eta:\eta \rightarrow \eta$ by $\alpha_{\eta}(u+s)=\alpha(u)+s$,  i.e.
\begin{displaymath}
\mathbf{\alpha_{\eta}} \triangleq
\left( \begin{array}{cc}
\alpha & 0 \\
0 & 1  \\
\end{array} \right).
\end{displaymath}
\begin{thm}
The triple $(\eta, [\cdot, \cdot]_{\theta}, \alpha_{\eta})$ is a multiplicative hom-Jordan-Lie algebra if and only if $\theta\in C_{\alpha}^{2}(V)$ satisfies $$d_{T}\theta=0.$$

We call the multiplicative hom-Jordan-Lie algebra  $(\eta, [\cdot, \cdot]_{\theta}, \alpha_{\eta})$ the central extension of $(L, [\cdot, \cdot]_{L}, \alpha)$ by the abelian hom-Jordan-Lie algebra $(R, 0, 1)$.
\end{thm}
\begin{proof}
First we show that $[\cdot, \cdot]_{\theta}$ satisfies (3),
\begin{eqnarray*}
&&-\delta[v+t, u+s]_{\theta}=-\delta([v, u]_{L}+\theta(v, u))\\
&=&[u, v]_{L}+\theta(u, v)=[u+s, v+t]_{\theta}.
\end{eqnarray*}

The fact that $\alpha_{\eta}$ is an algebra morphism with respect to the bracket $[\cdot, \cdot]_{\theta}$ follows from the fact that $\theta\circ\alpha=\theta$. More precisely,  we have
$$\alpha_{\eta}[u+s, v+t]_{\theta}=\alpha[u, v]_{L}+\theta(u, v).$$
On the other hand,  we have
$$[\alpha_{\eta}(u+s), \alpha_{\eta}(v+t)]_{\theta}=[\alpha(u)+s, \alpha(v)+t]_{\theta}=[\alpha(u),
\alpha(v)]_{L}+\theta(\alpha(u), \alpha(v)).$$

Since $\alpha$ is an algebra morphism and $\theta(\alpha(u), \alpha(v))=\theta(u, v)$,  we deduce that  $\alpha_{\eta}$ is an algebra morphism.

By direct computations,  we have
\begin{eqnarray*}
&&[\alpha_{\eta}(u+s), [v+t, w+m]_{\theta}]_{\theta}+c.p.(u+s, v+t, w+m)\\
&=&[\alpha(u)+s, [v, w]_{L}+\theta(v,w)]_{\theta}+c.p.(u+s, v+t, w+m)\\
&=&[\alpha(u), [v, w]_{L}]_{L}+\theta(\alpha(u), [v, w]_{L})+c.p.(u, v, w).
\end{eqnarray*}
By the hom-Jacobi identity of $L$,  $[\cdot, \cdot]_{\theta}$ satisfies the hom-Jacobi identity if and only if
$$\theta(\alpha(u), [v, w]_{L})+\theta(\alpha(v), [w, u]_{L})+\theta(\alpha(w), [u, v]_{L})=0. $$
Which exactly means that $d_{T}\theta=0.$ In fact,
\begin{eqnarray*}
&&d_{T}\theta(u, v, w)\\
&=&\delta^{3}(-\delta\theta([u, v]_{L}, \alpha(w))+\theta([u, w]_{L}, \alpha(v))-\delta\theta([v, w]_{L}, \alpha(u)))\\
&=&-(\theta([u, v]_{L}, \alpha(w))+\theta([w, u]_{L}, \alpha(v))+\theta([v, w]_{L}, \alpha(u)))\\
&=&\delta([\alpha_{\eta}(u+s), [v+t, w+m]_{\theta}]_{\theta}+c.p.)\\
&=&0.
\end{eqnarray*}

Then the triple $(\eta, [\cdot, \cdot]_{\theta}, \alpha_{\eta})$ is a multiplicative hom-Jordan-Lie algebra if and only if $\theta\in C_{\alpha}^{2}(V)$ satisfies $d_{T}\theta=0$.
\end{proof}
\begin{prop}
For $\theta_{1},  \theta_{2}\in Z^{2}(V)$,  if $\delta(\theta_{1}-\theta_{2})$ is exact,  the corresponding two central extensions $(\eta, [\cdot, \cdot]_{\theta_{1}}, \alpha_{\eta})$ and $(\eta, [\cdot, \cdot]_{\theta_{2}}, \alpha_{\eta})$ are isomorphic.
\end{prop}
\begin{proof}
Assume that  $\theta_{1}-\theta_{2}=\delta d_{T}f$,  $f\in C_{\alpha}^{1}(V)$. Thus we have
$$\theta_{1}(u, v)-\theta_{2}(u, v)=\delta d_{T}f(u, v)=-f([u, v]_{T}).$$
Define $f_{\eta}:\eta\rightarrow \eta$ by
$$f_{\eta}(u, s)=(u, s+f(u)).$$

If $(u, s+f(u))=0$, then $u=0$, $s=0$. Thus $f_{\eta}(0, 0)=0$,  so $f_{\eta}$ is injective.
If $(u, v)$ is the image,  then $f(u)$ is certain element. so $s=v-f(u)$ is uniquely determined,  Thus arbitrary items have the original item.
so $f_{\eta}$ is isomorphic.
It is straightforward to see that $f_{\eta}\circ\alpha_{\eta}=\alpha_{\eta}\circ f_{\eta}$.
Obviously,  $f_{\eta}$ is an isomorphism of vector spaces. we also have
\begin{eqnarray*}
&&f_{\eta}[(u, s), (v, t)]_{\theta_{1}}=f_{\eta}([u, v]_{L}, \theta_{1}(u, v))\\
&=&([u, v]_{L}, \theta_{1}(u, v)+f([u, v]_{L}))=([u, v]_{L}, \theta_{2}(u, v))\\
&=&[f_{\eta}(u, s), f_{\eta}(v, t)]_{\theta_{2}}.
\end{eqnarray*}
Therefore,  $f_{\eta}$ is also an isomorphism of multiplicative hom-Jordan-Lie algebras.
\end{proof}
\section{The adjniont representation of hom-Jordan-Lie algebras}
Let $(L, [\cdot, \cdot]_{L}, \alpha)$ be a regular hom-Jordan-Lie algebra. We consider that $L$ represents on itself via the bracket with respect to the morphism $\alpha$. A very interesting phenomenon is that the adjoint representation of a hom-Jordan-Lie algebra is not unique as one will see in sequel.
\begin{defn}
For any integer $s$,  the $\alpha^{s}$-adjoint representation of the regular  hom-Jordan-Lie algebra $(L, [\cdot, \cdot]_{L}, \alpha)$,  which we denote by $\ad_{s}$,  is defined by
$$\ad_{s}(u)(v)=\delta[\alpha^{s}(u), v]_{L}, \qquad \forall u, v\in L.$$
\end{defn}
\begin{lem}
With the above notations,  we have
$$\ad_{s}(\alpha(u))\circ \alpha=\alpha\circ \ad_{s}(u);$$
$$\ad_{s}([u, v]_{L})\circ \alpha=\ad_{s}(\alpha(u))\circ \ad_{s}(v)-\delta\ad_{s}(\alpha(v))\circ \ad_{s}(u).$$
Thus the definition of $\alpha^{s}$-adjoint representation is well defined.
\end{lem}
\begin{proof}
The conclusion follows from
\begin{eqnarray*}
\ad_{s}(\alpha(u))(\alpha(v))&=&\delta[\alpha^{s+1}(u), \alpha(v)]_{L}\\
&=&\alpha(\delta[\alpha^{s}(u), v]_{L})=\alpha\circ \ad_{s}(u)(v)
\end{eqnarray*}
and
\begin{eqnarray*}
\ad_{s}([u, v]_{L})(\alpha(w))&=&\delta[\alpha^{s}[u, v]_{L}, \alpha(w)]_{L}\\
&=&\delta[[\alpha^{s}(u), \alpha^{s}(v)]_{L}, \alpha(w)]_{L}\\
&=&-[\alpha(w), [\alpha^{s}(u), \alpha^{s}(v)]_{L}]_{L}\\
&=&[\alpha^{s+1}(u), [\alpha^{s}(v), w]_{L}]_{L}+[\alpha^{s+1}(v), [w, \alpha^{s}(u)]_{L}]_{L}\\
&=&[\alpha^{s+1}(u), [\alpha^{s}(v), w]_{L}]_{L}-\delta[\alpha^{s+1}(v), [\alpha^{s}(u), w]_{L}]_{L}\\
&=&\ad_{s}(\alpha(u))(\ad_{s}(v)(w))-\delta\ad_{s}(\alpha(v))(\ad_{s}(u)(w)).
\end{eqnarray*}
This completes the proof.
\end{proof}
The set of $k$-hom-cochains on $L$ with coefficients in $L$, which we denote by $C^{k}_{\alpha}(L;L)$ is given by
$$C^{k}_{\alpha}(L;L)=\{f\in C^{k}(L;L)|\alpha \circ f=f\circ \alpha\}.$$
  In particular,  the set of 0-hom-cochains are given by:
$$C^{0}_{\alpha}(L;L)=\{u\in L|\alpha(u)=u\}.$$

Associated to the $\alpha ^{s}$-adjoint representation,  the corresponding  operator
$$d_s:C_{\alpha}^{k}(L;L)
\rightarrow C_{\alpha}^{k+1}(L;L)(k=1,2)$$ is given by
\begin{equation}
d_{s}f(u_{1},u_{2})=\delta[\alpha^{1+s}(u_{1}),f(u_{2})]-[\alpha^{1+s}(u_{2}),f(u_{1})]-\delta f([u_{1}, u_{2}]);
\end{equation}
\begin{eqnarray*}
d_{s}f(u_{1},u_{2},u_{3})
&=&\delta[\alpha^{2+s}(u_{1}),f(u_{2},u_{3})]-[\alpha^{2+s}(u_{2}),f(u_{1},u_{3})]+\delta[\alpha^{2+s}(u_{3}),f(u_{1},u_{2})]\\
&&-f([u_{1}, u_{2}],\alpha(u_{3}))+\delta f([u_{1}, u_{3}],\alpha(u_{2}))-f([u_{2}, u_{3}],\alpha(u_{1})).
\end{eqnarray*}

For the $\alpha^{s}$-adjoint representation $\ad_{s}$,  we obtain the $\alpha ^{s}$-adjoint complex $(C_{\alpha}^{}(L;L), d_{s})$.

In the case of  hom-Lie algebras,  a 1-cocycle associated to the adjoint representation is a derivation. Similarly,  we have
\begin{prop}
Associated to the $\alpha^{s}$-adjoint representations $\ad_{s}$ of
the regular hom-Jordan-Lie algebra $(L, [\cdot, \cdot]_{L},
\alpha)$, It satisfies $\delta^{s+1}=1$, $D\in \alpha^{1}(L, L)$ is
a 1-cocycle if and only if D is an $\alpha^{s+1}$-derivation. i.e.
$D\in \Der_{\alpha^{s+1}}(L)$.
\end{prop}
\begin{proof}
The conclusion follows directly from the definition of the operator $d_{s}$. $D$ is closed if and only if
$$d_{s}(D)(u, v)=\delta[\alpha^{s+1}(u), D(v)]_{L}-[\alpha^{s+1}(v), D(u)]_{L}-\delta D[u, v]_{L}=0,$$
in the other word,
\begin{eqnarray*}
D[u, v]_{L}&=&-\delta[\alpha^{s+1}(v), D(u)]_{L}+[\alpha^{s+1}(u), D(v)]_{L}\\
&=&\delta^{s+1}([D(u), \alpha^{s+1}(v)]_{L}+[\alpha^{s+1}(u), D(v)]_{L}),
\end{eqnarray*}
which implies that $D$ is an $\alpha^{s+1}$-derivation.
\end{proof}

Let $\psi\in C_{\alpha}^{2}(L;L)$  be a bilinear operator commuting
with $\alpha$, also $\psi(u, v)=-\delta\psi(v,u )$. Consider a
t-parametrized family of bilinear operations
\begin{eqnarray}
[u, v]_{t}=[u, v]_{L}+t\psi(u, v).
\end{eqnarray}
Since $\psi$ commutes with $\alpha$,  $\alpha$ is a morphism with respect to the bracket $[\cdot, \cdot]_{t}$ for every $t$.
If all the brackets $[\cdot, \cdot]_{t}$ endow $(L, [\cdot, \cdot]_{t}, \alpha)$ regular hom-Jordan-Lie algebra structures,  we say that $\psi$ generates a deformation of the regular hom-Jordan-Lie algebra $(L, [\cdot, \cdot]_{L}, \alpha)$. By computing the hom-Jordan-Jacobi identity of $[\cdot, \cdot]_{t}$
\begin{eqnarray*}
&&[\alpha(u), [v, w]_{t}]_{t}+c.p.(u, v, w)\\
&=&[\alpha(u), [v, w]_{L}]_{t}+[\alpha(u), t\psi(v, w)]_{t}+c.p.(u, v, w)\\
&=&[\alpha(u), [v, w]_{L}]_{L}+t\psi(\alpha(u), [v, w]_{L})\\
&&+[\alpha(u), t\psi(v, w)]_{L}+t\psi(\alpha(u), t\psi(v, w))+c.p.(u, v, w)=0,
\end{eqnarray*}
this is equivalent to the conditions
\begin{equation}\psi(\alpha(u), \psi(v, w))+c.p.(u, v, w)=0;\end{equation}
\begin{equation}
(\psi(\alpha(u), [v, w]_{L})+[\alpha(u), \psi(v, w)_{L}]_{L})+c.p.(u, v, w)=0.
\end{equation}
Obviously,  (23) means that $\psi$ must itself define a hom-Jordan-Lie algebra structure on $L$. Furthermore, (24) means that $\psi$ is closed with respect to the $\alpha^{-1}$-adjoint representation $\ad_{-1}$,  i.e. $d_{-1}\psi=0$.
\begin{eqnarray*}
&&d_{-1}\psi(u, v, w)\\
&=&\delta[\alpha(u), \psi(v, w)]_{L}-[\alpha(v), \psi(u, w)]_{L}+\delta[\alpha(w), \psi(u, v)]_{L}\\
&&-\psi([u, v]_{L}, \alpha(w))+\delta\psi([u, w]_{L}, \alpha(v))-\psi([v, w]_{L}, \alpha(u))\\
&=&\delta[\alpha(u), \psi(v, w)]_{L}+\delta[\alpha(v), \psi(w, u)]_{L}+\delta[\alpha(w), \psi(u, v)]_{L}\\
&&+\delta\psi(\alpha(w), [u, v]_{L})+\delta\psi(\alpha(v), [w, u]_{L})+\delta\psi(\alpha(u), [v, w]_{L})\\
&=&\delta([\alpha(u), \psi(v, w)]_{L}+\psi(\alpha(u), [v, w]_{L})+c.p.(u, v, w))\\
&=&0.
\end{eqnarray*}

A deformation is  said to be trivial if there is a linear operator $N\in C_{\alpha}^{1}(L;L)$ such that for $T_{t}=\mathrm{id}+tN$,  there holds
\begin{eqnarray}
T_{t}[u, v]_{t}=[T_{t}(u), T_{t}(v)]_{L}.
\end{eqnarray}
\begin{defn}
A linear operator $N\in  C_{\alpha}^{1}(L, L)$ is called a hom-Nijienhuis operator if we have
\begin{eqnarray}
[Nu, Nv]_{L}=N[u, v]_{N},
\end{eqnarray}
where the bracket $[\cdot, \cdot]_{N}$ is defined by
\begin{eqnarray}
[u, v]_{N}\triangleq[Nu, v]_{L}+[u, Nv]_{L}-N[u, v]_{L}.
\end{eqnarray}
\end{defn}
\begin{thm}
Let $N\in C_{\alpha}^{1}(L, L)$ be a hom-Nijienhuis operator. Then a deformation of the regular hom-Jordan-Lie algebra $(L, [\cdot, \cdot]_{L}, \alpha)$ can be obtained by putting
$$\psi(u, v)=d_{-1}N(u, v)=[u, v]_{N}.$$
Furthermore,  this deformation is trivial.
\end{thm}
\begin{proof}
Since $\psi=d_{-1}N$,  $d_{-1}\psi=0$ is valid. To see that $\psi$ generates a deformation,  we need to check the hom-Jordan Jacobi identity for $\psi$. Using the explicit expression of  $\psi$,  we get
\begin{eqnarray*}&&\psi(\alpha(u), \psi(v, w))+c.p.(u, v, w)\\
&=&[\alpha(u), [v, w]_{N}]_{N}+c.p.(u, v, w)\\
&=&([\alpha(u), [Nv, w]_{L}+[v, Nw]_{L}-N[v, w]_{L}]_{N})+c.p.(u, v, w)\\
&=&([\alpha(u), [Nv, w]_{L}]_{N}+[\alpha(u), [v, Nw]_{L}]_{N}-[\alpha(u), N[v, w]_{L}]_{N})+c.p.(u, v, w)\\
&=&([N\alpha(u), [Nv, w]_{L}]_{L}+[\alpha(u), N[Nv, w]_{L}]_{L}-N[\alpha(u), [Nv, w]_{L}]_{L}\\
&&+[N\alpha(u), [v, Nw]_{L}]_{L}+[\alpha(u), N[v, Nw]_{L}]_{L}-N[\alpha(u), [v, Nw]_{L}]_{L}\\
&&-[N\alpha(u), N[v, w]_{L}]_{L}-[\alpha(u), N^{2}[v, w]_{L}]_{L}+N[\alpha(u), N[v, w]_{L}]_{L})+c.p.(u, v, w).
\end{eqnarray*}
Since
\begin{eqnarray*}&&[\alpha(u), N[Nv, w]_{L}]_{L}+[\alpha(u), N[v, Nw]_{L}]_{L}-[\alpha(u), N^{2}[v, w]_{L}]_{L}\\
&=&[\alpha(u), N([Nv, w]_{L}+[v, Nw]_{L}-N[v, w]_{L})]_{L}\\
&=&[\alpha(u), N[v, w]_{N}]_{L},
\end{eqnarray*}
we have
\begin{eqnarray*}&&\psi(\alpha(u), \psi(v, w))+c.p.(u, v, w)\\
&=&([N\alpha(u), [Nv, w]_{L}]_{L}-N[\alpha(u), [Nv, w]_{L}]_{L}+[N\alpha(u), [v, Nw]_{L}]_{L}\\
&&-N[\alpha(u), [v, Nw]_{L}]_{L}-[N\alpha(u), N[v, w]_{L}]_{L}+N[\alpha(u), N[v, w]_{L}]_{L}+[\alpha(u), N[v, w]_{N}]_{L})\\
&&+([N\alpha(v), [Nw, u]_{L}]_{L}-N[\alpha(v), [Nw, u]_{L}]_{L}+[N\alpha(v), [w, Nu]_{L}]_{L}\\
&&-N[\alpha(v), [w, Nu]_{L}]_{L}-[N\alpha(v), N[w, u]_{L}]_{L}+N[\alpha(v), N[w, u]_{L}]_{L}+[\alpha(v), N[w, u]_{N}]_{L})\\
&&+([N\alpha(w), [Nu, v]_{L}]_{L}-N[\alpha(w), [Nu, v]_{L}]_{L}+[N\alpha(w), [u, Nv]_{L}]_{L}\\
&&-N[\alpha(w), [u, Nv]_{L}]_{L}-[N\alpha(w), N[u, v]_{L}]_{L}+N[\alpha(w), N[u, v]_{L}]_{L}+[\alpha(w), N[u, v]_{N}]_{L})\\
&=&[N\alpha(u), [Nv, w]_{L}]_{L}+[N\alpha(v), [w, Nu]_{L}]_{L}\\
&&+[\alpha(w), N([u, v]_{N})]_{L}+c.p.(u, v, w)\\
&&+(N[\alpha(v), N[w, u]_{L}]_{L}-[N\alpha(v), N[w, u]_{L}]_{L})+c.p.(u, v, w)\\
&&-N[\alpha(u), [Nv, w]_{L}]_{L}-N[\alpha(w), [u, Nv]_{L}]_{L}+c.p.(u, v, w).
\end{eqnarray*}
Since $N$ commutes with $\alpha$,  by the hom-Jordan Jacobi identity of $L$,  we have
$$[N\alpha(u), [Nv, w]_{L}]_{L}+[\alpha(Nv), [w, Nu]_{L}]_{L}+[\alpha(w), [Nu, Nv]_{L}]_{L}=0.$$
Since $N$ is a hom-Nijienhuis operator,  we have
\begin{eqnarray*}
&&[N\alpha(u), [Nv, w]_{L}]_{L}+[\alpha(Nv), [w, Nu]_{L}]_{L}+[\alpha(w), N[u, v]_{N}]_{L}\\
&&+c.p.(u, v, w)=0.
\end{eqnarray*}
Furthermore,  also by the fact that $N$ is a hom-Nijienhuis operator,  we obtain
\begin{eqnarray*}
&&N[\alpha(v), N[w, u]_{L}]_{L}-[N\alpha(v), N[w, u]_{L}]_{L}+c.p.(u, v, w)\\
&=&N[\alpha(v), N[w, u]_{L}]_{L}-N[\alpha(v), [w, u]_{L}]_{N}+c.p.(u, v, w)\\
&=&-N[N\alpha(v), [w, u]_{L}]_{L}+N^{2}[\alpha(v), [w, u]_{L}]_{L})+c.p.(u, v, w).
\end{eqnarray*}
Thus by the hom-Jordan Jacobi identity of $L$,  we have
\begin{eqnarray*}
&&N[\alpha(v), N[w, u]_{L}]_{L}-[N\alpha(v), N[w, u]_{L}]_{L})+c.p.(u, v, w)\\
&=&-N[N\alpha(v), [w, u]_{L}]_{L}+c.p.(u, v, w).
\end{eqnarray*}
Therefore,  we have
\begin{eqnarray*}&&\psi(\alpha(u), \psi(v, w))+c.p.(u, v, w)\\
&=&-N[N\alpha(v), [w, u]_{L}]_{L}-N[\alpha(u), [Nv, w]_{L}]_{L}\\
&&-N[\alpha(w), [u, Nv]_{L}]_{L}+c.p.(u, v, w)\\
&=&-N([\alpha(Nv), [w, u]_{L}]_{L}+[\alpha(u), [Nv, w]_{L}]_{L}\\
&&+[\alpha(w), [u, Nv]_{L}]_{L})+c.p.(u, v, w)\\
&=&0.
\end{eqnarray*}
Thus $\psi$ generates a deformation of the hom-Jordan-Lie algebra $(L, [\cdot, \cdot]_{L}, \alpha)$.

Let $T_{t}=\mathrm{id}+tN$,  then we have
\begin{eqnarray*}
T_{t}[u, v]_{t}&=&(\mathrm{id}+tN)([u, v]_{L}+t\psi(u, v))\\
&=&(\mathrm{id}+tN)([u, v]_{L}+t[u, v]_{N})\\
&=&[u, v]_{L}+t([u, v]_{N}+N[u, v]_{L})+t^{2}N[u, v]_{N}.
\end{eqnarray*}
On the other hand,  we have
\begin{eqnarray*}
[T_{t}(u), T_{t}(v)]_{L}&=&[u+tNu, v+tNv]_{L}\\
&=&[u, v]_{L}+t([Nu, v]_{L}+[u, Nv]_{L})+t^{2}[Nu, Nv]_{L}.
\end{eqnarray*}
By(26),  (27),  we have
$$T_{t}[u, v]_{t}=[T_{t}(u), T_{t}(v)]_{L}, $$
which implies that the deformation is trivial.
\end{proof}

\section{$T$*-extensions of hom-Jordan-Lie algebras}
The method of the $T$*-extension was introduced in [6, 7] and the
$T$*-extension of an algebra is quadratic. The theory of quadratic
(color) hom-Lie algebras is referred to [1].
\begin{defn}
Let $(L,[\cdot,\cdot]_L,\alpha)$ be a hom-Jordan-Lie algebra. A bilinear form  $f$ on $L$ is said to be nondegenerate if
$$L^\perp=\{x\in L|f(x,y)=0, \forall y\in L\}=0;$$
invariant if
$$f([x,y],z)=f(x,[y,z]), \forall x,y,z\in L;$$
Jordansymmetric if
$$f(x,y)=f(y,x).$$
A subspace $I$ of $L$ is called isotropic if $I\subseteq I^\bot$.
\end{defn}

\begin{defn}
A bilinear form $f$ on a hom-Jordan-Lie algebra $(L,[\cdot,\cdot]_L,\alpha)$ is said to be Jordan consistent if $f$ satisfies
$$f(x,y)=0, \forall x, y\in L, |x|\neq|y|.$$
Throughout this section, we only consider Jordan consistent bilinear
forms.
\end{defn}

\begin{defn}
Let $(L,[\cdot,\cdot]_L,\alpha)$ be a hom-Jordan-Lie algebra over a field $\K$. If $L$ admits a nondegenerate invariant Jordansymmetric bilinear form $f$, then we call $(L,f,\alpha)$ a quadratic hom-Jordan-Lie algebra. In particular, a quadratic vector space $V$ is a  vector space admitting a nondegenerate Jordansymmetric bilinear form.

 Let $(L^{'},[\cdot,\cdot]_L',\beta)$ be another hom-Jordan-Lie algebra. Two quadratic hom-Jordan-Lie algebras $(L,f,\alpha)$ and
  $(L^{'},f',\beta)$ are said to be isometric if there exists a hom-Jordan-Lie algebra isomorphism $\phi: L\rightarrow L^{'}$ such that
   $f(x, y)=f'(\phi(x), \phi(y)), \forall x, y\in L$.
\end{defn}
\begin{lem}
Let $\ad$ be the adjoint representation of a hom-Jordan-Lie algebra $(L,[\cdot,\cdot]_{L},\alpha)$, and let us consider the linear map $\pi:L\rightarrow \mathrm{End}(L^{*})$($L^{*}$ is the dual space of L) defined by, $\pi(x)(f)(y)=-\delta f\circ \ad(x)(y),\forall x,y\in L$. Then $\pi$ is a representation of $L$ on $(L^{*},\tilde{\alpha})$ if and only if
\begin{equation}
\alpha\circ \ad\alpha(x)=\ad x\circ \alpha;
\end{equation}
\begin{equation}
\ad x\circ \ad\alpha(y)-\delta\ad y\circ \ad\alpha(x)=\alpha\circ \ad[x,y]_{L}.
\end{equation}
We call the representation $\pi$ the coadjoint representation of $L$, with $\tilde{\alpha}(f)=f\circ \alpha$.
\end{lem}
\begin{proof}
Firstly, we have
\begin{eqnarray*}
&&(\pi(\alpha(x))\circ \tilde{\alpha})(f)= -\delta \tilde{\alpha}(f)\circ \ad\alpha(x)= -\delta f\circ \alpha\circ \ad\alpha(x),
\end{eqnarray*}
and
\begin{eqnarray*}
&&\tilde{\alpha}(\pi(x))(f) =-\delta \tilde{\alpha}(f\circ \ad x)= -\delta f\circ \ad x\circ \alpha.
\end{eqnarray*}
Finally,
\begin{eqnarray*}
(\pi([x,y])\circ \tilde{\alpha})(f)= -\delta f\circ \alpha\circ \ad [x,y];
\end{eqnarray*}
\begin{eqnarray*}
&&(\pi(\alpha(x))\circ \pi(y)-\delta\pi(\alpha(y))\circ \pi(x))(f)\\
&=&-\delta \pi(\alpha(x))(f\circ ad y)+\pi(\alpha(y))(f\circ ad x)\\
&=&f\circ ad y\circ \ad\alpha(x)-\delta f\circ ad x\circ \ad\alpha(y)\\
&=&-\delta f\circ (\ad x\circ \ad\alpha(y)-\delta\ad y\circ \ad\alpha(x)).
\end{eqnarray*}
Then we have $$\pi(\alpha(x))\circ \tilde{\alpha}=\tilde{\alpha}(\pi(x));$$
$$\pi([x,y])\circ \tilde{\alpha}=\pi(\alpha(x))\circ \pi(y)-\delta\pi(\alpha(y))\circ \pi(x).$$
Then $\pi$ is a representation of $L$ on $(L^{*},\tilde{\alpha})$.
\end{proof}
\begin{lem}
Under the above notations, let $(L,[\cdot,\cdot]_L,\alpha)$ be a
hom-Jordan-Lie algebra, and $\omega:L\times L\rightarrow L^{*}$ be an
 bilinear map. Assume that the coadjoint representation
exists. The  space $L\oplus L^*$, provided with the
following bracket and a linear map defined respectively by
\begin{equation}
[x+f,y+g]_{L\oplus L^{\ast}}=[x,y]_{L}+\omega(x,y)+\delta\pi(x)g-\pi(y)f;
\end{equation}
\begin{equation}
\alpha^{'}(x+f)=\alpha(x)+f\circ\alpha.
\end{equation}
Then  $(L\oplus L^{*},[\cdot,\cdot]_{L\oplus L^{\ast}},\alpha^{'})$ is a hom-Jordan-Lie algebra if and only if $\omega$ is a 2-cocycle: $L\times L\rightarrow L^{*}$, i.e. $\omega\in Z^2(L, L^*)$.
\end{lem}
\begin{proof}
For any elements $x+f, y+g, z+h\in L\oplus L^{*}$.
%\begin{eqnarray*}
%&&[y+g,x+f]_{L\oplus L^{\ast}}\\
%&=&([y,x]_{L}+\omega(y,x)+\delta\pi(y)f_\pi(x)g)\\
%&=&-\delta[x,y]_{L}-\delta\omega(y,x)-\pi(x)g+\delta\pi(x)f
%&=&-\delta[x+f,y+g].
%\end{eqnarray*}
Note that
\begin{eqnarray*}
&&[\alpha^{'}(x+f),[y+g,z+h]_{L\oplus L^{\ast}}]_{L\oplus L^{\ast}}+c.p.(x+f,y+g,z+h)\\
&=&[\alpha(x),[y,z]_{L}]_{L}+c.p.(x,y,z)\\
&&+\omega(\alpha(x),[y,z]_{L})+\delta\pi(\alpha(x))\omega(y,z)+c.p(x,y,z)\\
&&+\pi(\alpha(x))(\pi(y)h)-\delta\pi(\alpha(x))(\pi(z)g)\\
&&-\pi([y,z]_{L})f\circ\alpha+c.p.(x+f,y+g,z+h).
\end{eqnarray*}
By the hom-Jordan Jacobi identity
$$[\alpha(x),[y,z]_{L}]_{L}+c.p.(x,y,z)=0.$$ On the
other hand
\begin{eqnarray*}
&&\pi(\alpha(x))(\pi(y)h)-\delta\pi(\alpha(x))(\pi(z)g)\\&&-\pi([y,z]_{L})f\circ\alpha+c.p.\\
&=&\pi(\alpha(y))(\pi(z)f)-\delta\pi(\alpha(z))(\pi(y)f)-\pi([y,z]_{L})f\circ\alpha\\
&&+\pi(\alpha(z))(\pi(x)g)-\delta\pi(\alpha(x))(\pi(z)g)-\pi([z,x]_{L})g\circ\alpha\\
&&+\pi(\alpha(x))(\pi(y)h)-\delta\pi(\alpha(y))(\pi(x)h)-\pi([x,y]_{L})h\circ\alpha.
\end{eqnarray*}
Since $\pi$ is the coadjoint representation of $L$, we have
\begin{eqnarray*}
&&\pi([x,y]_{L})h\circ\alpha\\
&=&-\delta h(\alpha\circ\ad([x,y]_{L}))\\
&=&-\delta h\circ\ad(x)\ad(\alpha(y))+h\circ\ad(y) \ad(\alpha(x))\\
&=&(\pi(x)h)(\ad(\alpha(y)))-\delta(\pi(y)h)(\ad(\alpha(x)))\\
&=&-\delta\pi(\alpha(y))(\pi(x)h)+\pi(\alpha(x))(\pi(y)h).
\end{eqnarray*}
Therefore, we have
\begin{eqnarray*}
\pi(\alpha(x))\pi(y)h-\delta\pi(\alpha(y))\pi(x)h-\pi([x,y]_{L})h\circ\alpha=0.
\end{eqnarray*}
Obviously,
$$\pi(\alpha(y))(\pi(z)f)-\delta\pi(\alpha(x))(\pi(z)g)-\pi([y,z]_{L})f\circ\alpha+c.p.=0.$$
Consequently,
$$[\alpha^{'}(x+f),[y+g,z+h]_{\Omega}]_{\Omega}+c.p.(x+f,y+g,z+h)=0,$$
if and only if
\begin{eqnarray*}
0&=&\delta\pi(\alpha(x))(\omega(y,z))-\pi(\alpha(y))(\omega(x,z))+\delta\pi(\alpha(z))(\omega(x,y))\\
&&+\omega(\alpha(x),[y,z]_{L})+\omega([x,z]_{L},\alpha(y))-\delta\omega([x,y]_{L},\alpha(z))\\
&=&[\alpha(x),\omega(y,z)]-\delta[\alpha(y),\omega(x,z)]+[\alpha(z),\omega(x,y)]\\
&&-\delta\omega([y,z]_{L},\alpha(x))+\omega([x,z]_{L},\alpha(y))-\delta\omega([x,y]_{L},\alpha(z))\\
&=&\delta d_{-1}\omega(x,y,z).
\end{eqnarray*}
That is $\omega\in Z^{2}(L,L^{*})$.
Then confirmation holds if and only if $\omega\in Z^2(L, L^*)$. Consequently, we prove the lemma.
\end{proof}

Clearly, $L^*$ is an abelian hom-ideal of $(L\oplus L^{*},[\cdot,\cdot]_{\alpha^{'}},\alpha^{'})$ and $L$ is isomorphic to the factor hom-Jordan-Lie algebra $(L\oplus L^{*})/L^*$. Moreover, consider the following Jordansymmetric bilinear form $q_{L}$ on $L\oplus L^{*}$ for all $x+f, y+g\in L\oplus L^*$,
$$q_{L}(x+f,y+g)=f(y)+g(x).$$
Then we have the following lemma.

\begin{lem}\label{lemma3.2}
Let $L$, $L^*$, $\omega$ and $q_L$ be as above. Then the triple $(L\oplus L^{*},q_L,\alpha^{'})$ is a quadratic hom-Jordan-Lie algebra if and only if $\omega$ is Jordancyclic in the following sense:
$$w(x,y)(z)=w(y,z)(x) ~~\text{for all}~ x,y,z\in L.$$
\end{lem}
\begin{proof}
If $x+f$ is orthogonal to  all elements of $L\oplus L^{*}$, then $f(y)=0$ and $g(x)=0$,  which implies that $x=0$ and $f=0$. So the Jordansymmetric bilinear form $q_{L}$ is nondegenerate.

Now suppose that  $x+f,y+g,z+h\in L\oplus L^*$, then
\begin{eqnarray*}
q_{L}([x+f,y+g]_{L\oplus L^{\ast}}, z+h)
&=&q_{L}([x,y]_{L}+\omega(x,y)+\delta\pi(x)g-\pi(y)f, z+h)\\
&=&\omega(x,y)(z)+\delta(\pi(x)g)(z)-\pi(y)(f)(z)+h([x,y]_{L})\\
&=&\omega(x,y)(z)-\delta g([x,z]_{L})+f([y,z]_{L})+h([x,y]_{L}).
\end{eqnarray*}
On the other hand,
\begin{eqnarray*}
q_{L}(x+f,[y+g,z+h]_{L\oplus L^{\ast}})
&=&q_{L}(x+f,[y,z]_{L}+\omega(y,z)+\delta\pi(y)h-\pi(z)g)\\
&=&f([y,z]_{L})+\omega(y,z)(x)+\delta\pi(y)(h)(x)-(\pi(z)g)(x)\\
&=&f([y,z]_{L})+\omega(y,z)(x)+h([x,y]_{L})-\delta g([x,z]_{L}).
\end{eqnarray*}
Hence the lemma follows.
\end{proof}

Now, for a Jordancyclic 2-cocycle $\omega$ we shall call the quadratic hom-Jordan-Lie algebra $(L\oplus L^{*},q_L,\alpha^{'})$ the $T^*$-extension of $L$ (by $\omega$) and denote the hom-Jordan-Lie algebra  $(L\oplus L^{*},[\cdot,\cdot]_{L\oplus L^{\ast}},\alpha^{'})$ by $T_\omega^*L$.

\begin{defn}
Let $L$ be a hom-Jordan-Lie algebra over a field $\K$. We inductively define a derived series
$$(L^{(n)})_{n\geq 0}: L^{(0)}=L,\ L^{(n+1)}=[L^{(n)},L^{(n)}],$$
a central descending series
$$(L^{n})_{n\geq 0}: L^{0}=L,\ L^{n+1}=[L^{n},L],$$
and a central ascending series
$$(C_{n}(L))_{n\geq 0}: C_{0}(L)=0, C_{n+1}(L)=C(C_{n}(L)),$$
where $C(I)=\{a\in L| [a,L]\subseteq I\}$ for a subspace $I$ of $L$.

$L$ is called solvable and nilpotent(of length $k$) if and only if there is a (smallest) integer $k$ such that $L^{(k)}=0$ and $L^{k}=0$, respectively.
\end{defn}

In the following theorem we discuss some properties of  $T_\omega^*L$.

\begin{thm}
Let $(L,[\cdot,\cdot]_{L},\alpha)$ be a hom-Jordan-Lie algebra over a field $\K$.
\begin{enumerate}[(1)]
   \item  If $L$ is solvable (nilpotent) of length $k$, then the $T^{*}$-extension
          $T^{*}_{\omega}L$ is solvable (nilpotent) of length $r$, where $k\leq r\leq k+1$ $(k\leq r\leq2k-1)$.
   \item  If $L$ is decomposed into a direct sum of two  hom-ideals of $L$, so is the trivial $T^{*}$-extension $T^{*}_{0}L$.
\end{enumerate}
\end{thm}
\begin{proof}
(1) Firstly we suppose that $L$ is solvable of length $k$. Since
$(T^{*}_{\omega}L)^{(n)}/L^{*}\cong L^{(n)}$ and $L^{(k)}=0$, we have $(T^{*}_{\omega}L)^{(k)}\subseteq L^{*}$, which implies $(T^{*}_{\omega}L)^{(k+1)}=0$ because $L^{*}$ is abelian, and it follows that $T^{*}_{\omega}L$ is solvable of length $k$ or $k+1$.

Suppose now that $L$ is nilpotent of length $k$. Since $(T^{*}_{\omega}L)^{n}/L^{*}\cong L^{n}$ and $L^{k}=0$, we have
$(T^{*}_{\omega}L)^{k}\subseteq L^{*}$. Let $\beta\in(T^{*}_{\omega}L)^{k}\subseteq L^{*}, b\in L$, $x_{1}+f_{1}, \cdots, x_{k-1}+f_{k-1}\in T^{*}_{\omega}L$, $ 1\leq i\leq k-1$, we have
\begin{eqnarray*}
&&[[\cdots[\beta,x_1+f_{1}]_{L\oplus L^{\ast}},\cdots]_{L\oplus L^{\ast}},x_{k-1}+f_{k-1}]_{L\oplus L^{\ast}}(b)\\
&=&\delta^{k-1}\beta\ad x_1\cdots\ad x_{k-1}(b)=\beta([x_1,[\cdots,[x_{k-1},b]_{L}\cdots]_{L}]_{L})\in\beta(L^{k})=0.
\end{eqnarray*}
This proves that $(T^{*}_{\omega}L)^{2k-1}=0$. Hence $T^{*}_{w}L$ is nilpotent of length at least $k$ and at most $2k-1$.

(2) Suppose that $0\neq L=I\oplus J$,  where $I$ and $J$ are two nonzero hom-ideals of $(L[\cdot,\cdot]_{L},\alpha)$. Let $I^{*}$ (resp. $J^{*}$) denote the subspace of all
linear forms in $L^{*}$  vanishing on $J$ (resp. $I$). Clearly, $I^{*}$ (resp. $J^{*}$) can canonically be identified with the dual space of $I$ (resp. $J$) and $L^*\cong I^*\oplus J^*$.

Since $[I^*,L]_{L\oplus L^{\ast}}(J)=I^*([L,J]_{L})\subseteq I^*(J)=0$ and $[I,L^*]_{L\oplus L^{\ast}}(J)=L^*([I,J]_{L})\subseteq L^*(I\cap J)=0$,  we have $[I^*,L]_{L\oplus L^{\ast}}\subseteq I^*$ and $[I,L^*]_{L\oplus L^{\ast}}\subseteq I^*$. Then
\begin{eqnarray*}[T^{*}_{0}I,T^{*}_{0}L]_{L\oplus L^{\ast}}&=&[I\oplus I^*,L\oplus L^*]_{L\oplus L^{\ast}}\\
&=&[I,L]_{L}+[I,L^*]_{L\oplus L^{\ast}}+[I^*,L]_{L\oplus L^{\ast}}+[I^*,L^*]_{L\oplus L^{\ast}}\subseteq I\oplus I^*=T^{*}_{0}I.
\end{eqnarray*}
$T^{*}_{0}I$ is a hom-ideal of $L$ and so is $T^{*}_{0}J$ in the same way. Hence $T^{*}_{0}L$ can be decomposed into the direct sum $T^{*}_{0}I\oplus T^{*}_{0}J$ of two nonzero hom-ideals of $T^{*}_{0}L$.
\end{proof}

In the proof of a criterion for recognizing $T^*$-extensions of a hom-Jordan-Lie algebra, we will need the following result.

\begin{lem}\label{lemma3.1}
Let $(L,q_{L},\alpha)$ be a quadratic hom-Jordan-Lie algebra of even dimension $n$ over a field $\K$ and $I$ be an isotropic $n/2$-dimensional subspace of $L$. If $I$ is a hom-ideal of $(L,[\cdot,\cdot]_{L},\alpha)$, then $I$ is abelian.
\end{lem}
\begin{proof}
Since dim$I$+dim$I^{\bot}=n/2+\dim I^{\bot}=n$ and $I\subseteq I^{\bot}$, we have $I=I^{\bot}$.
If $I$ is a ideal of $(L,[\cdot,\cdot]_{L},\alpha)$, then $q_{L}(L,[I,I^{\bot}])=q_{L}([L,I],I^{\bot})\subseteq q_{L}(I,I^{\bot})=0$, which implies $[I,I]=[I,I^{\bot}]\subseteq L^{\bot}=0$.
\end{proof}

\begin{thm}
Let $(L,q_{L},\alpha)$ be a quadratic  hom-Jordan-Lie algebra, of even dimension $n$ over a field $\K$ of characteristic not equal to two. Then $(L,q_{L},\alpha)$ is isometric to a $T^{*}$-extension $(T_{\omega}^{*}B,q_{B},\beta^{'})$ if and only if $n$ is even and $(L,[\cdot,\cdot]_{L},\alpha)$ contains an isotropic hom-ideal $I$ of dimension $n/2$. In particular, $B\cong L/I$.
\end{thm}
\begin{proof}
($\Longrightarrow$) Since dim$B$=dim$B^{*}$, dim$T^{*}_{\omega}B$ is even. Moreover, it is clear that $B^{*}$ is a hom-ideal of half the dimension of $T^{*}_{\omega}B$ and by the definition of $q_{B}$, we have $q_B(B^*,B^*)=0$, i.e., $B^*\subseteq (B^*)^\bot$ and so $B^*$ is isotropic.

($\Longleftarrow$) Suppose that $I$ is an $n/2$-dimensional isotropic hom-ideal of $L$. By Lemma 7.9,  $I$ is abelian. Let $B=L/I$ and $p: L \rightarrow B$ be the canonical projection. Clearly, $|p(x)|=|x|, \forall x\in L$. Since $\ch \K\neq2$,  we can choose an isotropic complement subspace $B_{0}$ to $I$ in $L$, i.e., $L=B_{0}\dotplus I$ and $B_{0}\subseteq B_{0}^{\bot}$. Then $B_{0}^{\bot}=B_{0}$ since dim$B_0=n/2$.

Denote by $p_{0}$ (resp. $p_{1}$) the projection $L \rightarrow B_{0}$ (resp. $L\rightarrow I$) and let $q_{L}^{*}$
 denote the homogeneous linear map $I \rightarrow B^{*}: i \mapsto q_{L}^{*}(i)$, where $q_{L}^{*}(i)(p(x)):= q_{L}(i,x), \forall x\in L$.
 We claim that $q_{L}^{*}$ is a linear isomorphism. In fact, if $p(x)=p(y)$, then $x-y\in I$, hence $q_{L}(i,x-y)\in q_{L}(I,I)=0$ and
 so $q_{L}(i,x)=q_{L}(i,y)$, which implies $q_{L}^{*}$ is well-defined and it is easily seen that $q_{L}^{*}$ is linear. If
 $q_{L}^{*}(i)=q_{L}^{*}(j)$, then $q_{L}^{*}(i)(p(x))=q_{L}^{*}(j)(p(x)), \forall x\in L$, i.e., $q_{L}(i,x)=q_{L}(j,x)$,
 which implies $i-j\in L^\bot=0$, hence $q_{L}^{*}$ is injective. Note that $\dim I=\dim B^*$, then $q_{L}^{*}$ is surjective.

In addition, $q_{L}^{*}$ has the following property:
\begin{eqnarray*}
q_{L}^{*}([x,i])(p(y))&=&q_{L}([x,i]_{L},y)=-\delta([i,x]_{L},y)=-\delta q_{L}(i,[x,y]_{L})\\
&=&-\delta q_{L}^{*}(i)p([x,y]_{L})=-\delta q_{L}^{*}(i)[p(x),p(y)]_{L}\\
&=&- q_{L}^{*}(i)(\ad p(x)(p(y)))=\delta(\pi(p(x))q_{L}^{*}(i))(p(y))\\
&=&[p(x),q_{L}^{*}(i)]_{L\oplus L^{*}}(p(y)),
\end{eqnarray*}
where $x,y\in L$, $i\in I$. A similar computation shows that
$$q_{L}^{*}([x,i])=[p(x),q_{L}^{*}(i)]_{L\oplus L^{*}},\quad q_{L}^{*}([i,x])=[q_{L}^{*}(i),p(x)]_{L\oplus L^{*}}.$$
Define a homogeneous bilinear map
\begin{eqnarray*}
\omega:~~~~~ B\times B~~~~&\longrightarrow&B^{*}\\
(p(b_0),p(b_0'))&\longmapsto&q_{L}^{*}(p_{1}([b_0,b_0'])),
\end{eqnarray*}
where $b_0,b_0'\in B_{0}.$ Then $|w|=0$ and $w$ is well-defined since the restriction of the projection $p$ to $B_{0}$ is a linear isomorphism.

Let $\varphi$ be the linear map $L \rightarrow B\oplus B^{*}$ defined by $\varphi(b_{0}+i)=p(b_{0})+q_{L}^{*}(i), \forall b_0+i\in B_0\dotplus I=L. $
 Since the  restriction of $p$  to $B_{0}$ and $q_{L}^{*}$ are linear isomorphisms, $\varphi$ is also a linear isomorphism. Note that
\begin{eqnarray*}
&&\varphi([b_0+i,b_0'+i']_{L})=\varphi([b_0,b_0']_{L}+[b_0,i']_{L}+[i,b_0']_{L})\\
&=&\varphi(p_{0}([b_0,b_0']_{L})+p_{1}([b_0,b_0']_{L})+[b_0,i']_{L}+[i,b_0']_{L})\\
&=&p(p_{0}([b_0,b_0']_{L}))+q_{L}^{*}(p_{1}([b_0,b_0']_{L})+[b_0,i']_{L}+[i,b_0']_{L})\\
&=&[p(b_0),p(b_0')]_{L}+\omega(p(b_0),p(b_0'))+[p(b_0),q_{L}^{*}(i')]_{L}+[q_{L}^{*}(i),p(b_0')]_{L}\\
&=&[p(b_0),p(b_0')]_{L}+\omega(p(b_0),p(b_0'))+\delta\pi(p(b_0)(q_{L}^{*}(i'))-\pi(p(b_0')(q_{L}^{*}(i))\\
&=&[p(b_0)+q_{L}^{*}(i),p(b_0')+q_{L}^{*}(i')]_{B\oplus B^{*}}\\
&=&[\varphi(b_0+i),\varphi(b_0'+i')]_{L\oplus L^{*}}.
\end{eqnarray*}
Then $\varphi$ is an isomorphism of  algebras, and so $(B\oplus B^{*},[\cdot,\cdot]_{B\oplus B^{*}},\beta)$ is a hom-Jordan-Lie algebra.
Furthermore, we have
{\setlength\arraycolsep{2pt}
\begin{eqnarray*}
q_{B}(\varphi(b_{0}+i),\varphi(b_{0}'+i'))&=&q_{B}(p(b_{0})+q_{L}^{*}(i),p(b_{0}')+q_{L}^{*}(i'))\\
&=&q_{L}^{*}(i)(p(b_{0}'))+q_{L}^{*}(i')(p(b_{0}))\\
&=&q_{L}(i,b_{0}')+q_{L}(i',b_{0})\\
&=&q_{L}(b_{0}+i,b_{0}'+i'),
\end{eqnarray*}}
then $\varphi$ is isometric. The relation
\begin{eqnarray*}
q_{B}([\varphi(x),\varphi(y)],\varphi(z))&=&q_B(\varphi([x,y]),\varphi(z))\\
&=&q_{L}([x,y],z)=q_{L}(x,[y,z])=q_{B}(\varphi(x),[\varphi(y),\varphi(z)]).
\end{eqnarray*}
 Which implies that $q_B$ is a nondegenerate invariant Jordan-symmetric bilinear form, and so $(B\oplus B^{*}, q_B,\beta^{'})$ is a quadratic hom-Jordan-Lie algebra.
In this way, we get a $T^*$-extension $T^{*}_{\omega}B$ of $B$ and consequently, $(L,q_{L},\alpha)$ and $(T^{*}_{\omega}B,q_{B},\beta^{'})$ are isometric as required.
\end{proof}

Let$(L,[\cdot,\cdot]_{L},\alpha)$ be a hom-Jordan-Lie algebra over a field $\K$, and let $\omega_{1}: L\times L\rightarrow L^{*}$
and $\omega_{2}: L \times L\rightarrow L^{*}$ be two different Jordancyclic 2-cocycles satisfying $|\omega_1|=|\omega_2|=0$. The
 $T^{*}$-extensions $T^{*}_{\omega_{1}}L$ and $T^{*}_{w_{2}}L$ of $L$ are said to be  equivalent if there exists an isomorphism
 of hom-Jordan-Lie algebras  $\phi: T^{*}_{\omega_{1}}L\rightarrow  T^{*}_{\omega_{2}}L$ which is the identity on the hom-ideal $L^{*}$
  and which induces the identity on the factor hom-Jordan-Lie algebra $T^{*}_{\omega_{1}}L/L^{*}\cong L\cong T^{*}_{\omega_{2}}L/L^{*}.$ The
  two $T^{*}$-extensions $T^{*}_{\omega_{1}}L$ and $T^{*}_{\omega_{2}}L$ are said to be  isometrically equivalent if  they are equivalent and $\phi$ is an isometry.
\begin{prop}\label{proposition3.1}
Let $L$ be a hom-Jordan-Lie algebra over a field $\K$ of characteristic not equal to 2, and $\omega_{1}$, $\omega_{2}$ be two Jordancyclic 2-cocycles $L\times L\rightarrow L^{*}$ satisfying $|\omega_i|=0$. Then we have
\begin{enumerate}[(i)]
   \item  $T^{*}_{\omega_{1}}L$ is equivalent to  $T^{*}_{\omega_{2}}L$ if and only if there is $z\in C^1(L, L^{*})$ such that
    \begin{equation}\label{equation3.1}
       \omega_{1}(x, y)-\omega_{2}(x, y)=\delta\pi(x)z(y)-\pi(y)z(x)-z([x,y]_{L}),  \forall x, y\in L.
    \end{equation}

        If this is the case, then the Jordansymmetric part $z_{s}$ of $z$, defined by $z_{s}(x)(y):=\frac{1}{2}(z(x)(y)+z(y)(x))$,
        for all $x,y\in L$, induces a Jordansymmetric invariant bilinear form on $L$.
   \item  $T_{\omega_{1}}^{*}L$ is isometrically equivalent to $T_{\omega_{2}}^{*}L$ if and only if there is $z\in C^1(L, L^{*})$ such that $(33)$ holds for all ~$x,y\in L$ and the Jordansymmetric part $z_{s}$ of $z$ vanishes.
\end{enumerate}
\end{prop}
\begin{proof}
(i) $T^{*}_{\omega_{1}}L$ is equivalent to  $T^{*}_{\omega_{2}}L$ if and only if there is an isomorphism of hom-Jordan-Lie algebras $\Phi: T_{\omega_{1}}^{*}L\rightarrow T_{\omega_{2}}^{*}L$ satisfying $\Phi|_{L^*}=1_{L^*}$ and $x-\Phi(x)\in L^*, \forall x\in L$.

Suppose that $\Phi: T_{\omega_{1}}^{*}L\rightarrow T_{\omega_{2}}^{*}L$ is an isomorphism of hom-Jordan-Lie algebra and define a linear map $z: L\rightarrow L^{*}$ by $z(x):=\Phi(x)-x$, then $z\in C^1(L, L^{*})$ and for all $x+f,y+g\in T^{*}_{\omega_{1}}L$, we have
\begin{eqnarray*}
&&\Phi([x+f,y+g]_{\Omega})\\
&=&\Phi([x,y]_{L}+\omega_{1}(x,y)+\delta\pi(x)g-\pi(y)f)\\
&=&[x,y]_{L}+z([x,y]_{L})+\omega_{1}(x,y)+\delta\pi(x)g-\pi(y)f.
\end{eqnarray*}
On the other hand,
\begin{eqnarray*}
&&[\Phi(x+f),\Phi(y+g)]\\
&=&[x+z(x)+f,y+z(y)+g]\\
&=&[x,y]_{L}+\omega_2(x,y)+\delta\pi(x)g+\delta\pi(x)z(y)-\pi(y)z(x)-\pi(y)f.
\end{eqnarray*}
Since $\Phi$ is an isomorphism, (32) holds.

Conversely, if there exists $z\in C^{1}(L, L^{*})$ satisfying (30), then we can define
 $\Phi: T_{\omega_{1}}^{*}L\rightarrow T_{\omega_{2}}^{*}L$ by $\Phi(x+f):=x+z(x)+f$. It is easy to
 prove that $\Phi$ is an isomorphism of hom-Jordan-Lie algebras such that $\Phi|_{L^*}=\Id_{L^*}$ and $x-\Phi(L)\in L^*, \forall x\in L$,
 i.e. $T^{*}_{\omega_{1}}L$ is equivalent to  $T^{*}_{\omega_{2}}L$.

Consider the Jordanymmetric bilinear form $q_{L}: L\times L\rightarrow \K, (x,y)\mapsto z_s(x)(y)$ induced by $z_s$. Note that
\begin{eqnarray*}
&&\omega_{1}(x, y)(m)-\omega_{2}(x, y)(m)\\
&=&\delta\pi(x)z(y)(m)-\pi(y)z(x)(m)-z([x,y]_{L})(m)\\
&=&-\delta z(y)([x,m]_{L})+z(x)([y,m]_{L})-z([x,y]_{L})(m),
\end{eqnarray*}
and
\begin{eqnarray*}
&&\omega_{1}(y,m)(x)-\omega_{2}(y,m)(x)\\
&=&\delta\pi(y)z(m)(x)-\pi(m)z(y)(x)-z([y,m]_{L}(x))\\
&=&z(m)([x,y]_{L})-\delta z(y)([x,m]_{L})-z([y,m]_{L})(x).
\end{eqnarray*}
Since both $\omega_{1}$ and $\omega_{2}$ are Jordancyclic, the right hand sides of above two equations are equal. Hence
\begin{eqnarray*}
&&-\delta z(y)([x,m]_{L})+z(x)([y,m]_{L})-z([x,y]_{L})(m)\\
&=&z(m)([x,y]_{L})-\delta z(y)([x,m]_{L})-z([y,m]_{L})(x).
\end{eqnarray*}
That is,
\begin{eqnarray*}
z(x)([y,m]_{L})+z([y,m]_{L})(x)=z([x,y]_{L})(m)+z(m)([x,y]_{L}).
\end{eqnarray*}
Since ch$\K\neq2$, $q_{L}(x,[y,m])=q_{L}([x,y],m)$, which proves the invariance of the Jordansymmetric bilinear form $q_{L}$ induced by $z_s$.

(ii) Let the isomorphism $\Phi$ be defined as in (i). Then for all $x+f, y+g\in L\oplus L^{*}$, we have
\begin{eqnarray*}
&&q_{B}(\Phi(x+f),\Phi(y+g))=q_{B}(x+z(x)+f,y+z(y)+g)\\
&=&z(x)(y)+f(y)+z(y)(x)+g(x)\\
&=&z(x)(y)+z(y)(x)+f(y)+g(x)\\
&=&2z_s(x)(y)+q_{B}(x+f, y+g).
\end{eqnarray*}
Thus, $\Phi$ is an isometry if and only if $z_{s}=0$.
\end{proof}

\end{document}